\documentclass[11pt]{article}

\usepackage{mathpazo}
\usepackage{amsfonts}
\usepackage{amsmath}
\usepackage{graphicx}
\usepackage{latexsym}
\usepackage{mathabx}

\usepackage[margin=1.05in]{geometry}
  
\usepackage[T1]{fontenc}
\usepackage{times}
\usepackage{color,graphicx}
\usepackage{array}
\usepackage{enumerate}
\usepackage{amsmath}
\usepackage{amssymb}
\usepackage{amsthm}
\usepackage{pgfplots}
\usepackage{pgf}
\usepackage{tikz}
\usetikzlibrary{patterns}
\usepgfplotslibrary{patchplots} 
\usetikzlibrary{pgfplots.patchplots} 
\pgfplotsset{width=9cm,compat=1.5.1}

\usepackage{amsthm}

\newtheorem{lemma}{Lemma}
\newtheorem{theorem}{Theorem}
\newtheorem{corollary}{Corollary}

\newtheorem{proposition}{Proposition}

\newtheorem{example}{Example}

\newcommand{\be}{\mathbf e}

\usepackage{xcolor}
\usepackage{makeidx}
\usepackage[colorlinks=true,linkcolor=blue,anchorcolor=blue,citecolor=red,urlcolor=magenta]{hyperref}
\usepackage[alphabetic,backrefs]{amsrefs}

\begin{document}
\title{Fractional revival between twin vertices}

\author{
	Hermie Monterde,\textsuperscript{\!\!1}
}

\maketitle

%
%
%









\begin{abstract}
In this paper, we provide a characterization of fractional revival between twin vertices in a weighted graph with respect to its adjacency, Laplacian and signless Laplacian matrices. As an application, we characterize fractional revival between apexes of double cones.
\end{abstract}

	


\addtocounter{footnote}{1}
\footnotetext{Department of Mathematics, University of Manitoba, Winnipeg, MB, Canada R3T 2N2}

\tableofcontents

\section{Introduction}\label{secINTRO}

Continuous-time quantum walks were introduced by Farhi and Gutmann in 1998 \cite{Farhi1998}, and since then have become central in the study of various types of state transfer in quantum spin networks. In quantum information theory, continuous-time quantum walks that exhibit types of state transfer that allow for the generation of entanglements is considered desirable. One such type of state transfer is fractional revival.

Fractional revival from a mathematical standpoint had been studied under adjacency dynamics \cite{Chan2019,CHAN20221}, Laplacian dynamics \cite{Chan2020}, in association schemes \cite{CHAN20201}, abelian Cayley graphs \cite{Cao2022}, threshold graphs \cite{Kirkland2020} and non-cospectral vertices \cite{Godsil2022}. In this paper, we study fractional revival between twin vertices with respect to the adjacency, Laplacian and signless Laplacian matrices.

The outline of this paper is as follows. In Section \ref{secTV}, we review a handful of known state transfer properties of twin vertices and show that proper fractional revival is monogamous when it involves a vertex with a twin. In Section \ref{secFR}, we provide characterizations (Theorems \ref{frtwinschar} and \ref{anotherchar}) as well as necessary and sufficient conditions (Corollaries \ref{frtwinschar1} and \ref{frtwinschar2}) for fractional revival to occur between twin vertices. In Section \ref{secFR1}, we explore the existence of periodicity, perfect state transfer, and pretty good state transfer between twin vertices that admit proper fractional revival. In Section \ref{secCharPoly}, we explore what happens to our results in Sections \ref{secFR} and \ref{secFR1} if we add that $\phi(M,t)\in\mathbb{Z}[x]$. We then apply our results in Section \ref{secDC} to obtain a characterization of double cones that admit proper and balanced fractional revival. The remainder of this section is devoted to relevant graph and matrix theory background, as well as basic notions of state transfer.

Throughout this paper, we assume that $X$ is a connected weighted undirected graph with possible loops but no multiple edges, and we denote its vertex set by $V(X)$. We allow the edges of $X$ to have non-zero real weights. We say that $X$ is \textit{simple} if $X$ has no loops, and $X$ is \textit{unweighted} if all edges of $X$ have weight one. For $u\in V(X)$, we denote the characteristic vector of $u$ as $\textbf{e}_u$, which is a vector with a $1$ on the entry indexed by $u$ and $0$'s elsewhere, and the set of neighbours of $u$ by $N_X(u)$. We also represent the characteristic polynomial of a square matrix $H$ in the variable $t$ by $\phi(H,t)$. Lastly, we denote the simple unweighted empty, cycle, complete, and path graphs on $n$ vertices as $O_n$, $C_n$, $K_n$, and $P_n$, respectively.

The adjacency matrix $A(X)$ of $X$ is defined entrywise as $A(X)_{u,v}=\omega_{u,v}$ if $u$ is adjacent to $v$, where $\omega_{u,v}$ is the weight of the edge $(u,v)$, and $A(X)_{u,v}=0$ otherwise. The degree matrix $D(X)$ of $X$ is the diagonal matrix of vertex degrees of $X$, where $\operatorname{deg}(u)=2\omega_{u,u}+\sum_{j\neq u}\omega_{u,j}$ for each $u\in V(X)$. The Laplacian matrix $L(X)$ of $X$ is the matrix $L(X)=D(X)-A(X)$, while the signless Laplacian matrix $Q(X)$ of $X$ is the matrix $Q(X)=D(X)+A(X)$. We use $M(X)$ to denote $A(X)$, $L(X)$ or $Q(X)$. If the context is clear, then we simply write $M(X)$, $A(X)$, $L(X)$, $Q(X)$ and $D(X)$ as $M$, $A$, $L$, $Q$ and $D$, resp.

The \textit{(continuous-time) quantum walk} on $X$ with respect to $H$ is determined by the transition matrix
\begin{equation}
\label{M}
U_H(t)=e^{itH}.
\end{equation}
The matrix $H$ is called the \textit{Hamiltonian} of the quantum walk. Typically, $H$ is taken to be $A$, $L$, or $Q$, but in general, any Hermitian matrix $H$ that respects the adjacencies of $X$ should work (that is, $H_{u,v}=0$ if and only if there is no edge between $u$ and $v$). If $H=M$ and $M$ is clear from the context, then we simply write $U_M(t)$ as $U(t)$. Since $M$ is real symmetric, we can write $M$ in its spectral decomposition as $M=\sum_{j}\lambda_jE_j$, where the $\lambda_j$'s are the distinct eigenvalues of $M$ and $E_j$ is the orthogonal projection matrix onto the eigenspace associated with $\lambda_j$. If the eigenvalues are not indexed, then we also denote by $E_{\lambda}$ the orthogonal projection matrix corresponding to the eigenvalue $\lambda$ of $M$. This allows us to write (\ref{M}) as
\begin{equation}
\label{specdecM}
U(t)=\sum_{j}e^{it\lambda_j}E_j.
\end{equation}
Note that $U(t)$ is a complex symmetric unitary matrix, and so $\sum_{j=1}|U(\tau)_{u,j}|^2=1$ for any vertex $u$ of $X$. For this reason, if $u$ and $v$ are vertices of $X$, then $|U(\tau)_{u,v}|^2$ is interpreted as the fidelity (probability) of quantum state transfer between $u$ and $v$ at time $\tau$. It is known that if $X$ is simple and weighted $k$-regular, the quantum walks with respect to $A$, $L$, and $Q$ are equivalent up to a global phase, and so they all exhibit the same types of state transfer. Similarly, if $X$ is bipartite, then the quantum walks with respect to $L$ and $Q$ are equivalent up to a local phase, and so they again exhibit the same types of state transfer.

We say that \textit{perfect state transfer} (PST) occurs between $u$ and $v$ at time $\tau$ if $|U(\tau)_{u,v}|^2=1$. We say that $u$ is \textit{periodic} at time $\tau$ if $|U(\tau)_{u,u}|^2=1$. We say that \textit{pretty good state transfer} occurs between $u$ and $v$ if for every $\epsilon>0$, there exists a time $\tau_{\epsilon}$ such that $|U(\tau_{\epsilon})_{u,v}|^2>1-\epsilon$. We say that $(\alpha,\beta)$-\textit{fractional revival} (FR) occurs from $u$ to $v$ at time $\tau$, where $\alpha=|U(\tau)_{u,u}|$ and $\beta=|U(\tau)_{u,v}|$, if $|U(\tau)_{u,u}|^2+|U(\tau)_{u,v}|^2=1$. If the particular $\alpha$ and $\beta$ are not important, then we simply say fractional revival. The case $\beta\neq 0$ is called \textit{proper} FR, while the case $|\alpha|=|\beta|$ is called \textit{balanced} FR. Note that FR is a generalization of PST and periodicity. We sometimes say adjacency PST (resp., periodicity, FR) when we talk about PST (resp., periodicity, FR) whenever $M=A$; similar language applies when $M=L,Q$.

\section{Twin vertices}\label{secTV}

Two distinct vertices $u$ and $v$ of $X$ are \textit{twins} if (i) $N_X(u)\backslash \{u,v\}=N_X(v)\backslash \{u,v\}$, (ii) the edges $(u,w)$ and $(v,w)$ have the same weight for each $w\in N_X(u)\backslash \{u,v\}$, and (iii) if there are loops on $u$ and $v$, then they have the same weight. Given $\omega,\eta\in\mathbb{R}$, a subset $T=T(\omega,\eta)$ of $V(X)$ with at least two vertices is a \textit{set of twins} in $X$ if the vertices in $T$ are pairwise twins. That is, each vertex in $T$ has a loop of weight $\omega$ (which is absent if $\omega=0$), and every pair of vertices in $T$ form an edge with weight $\eta$ (which is absent if $\eta=0$).

We now state a spectral characterization of twin vertices due to Monterde \cite[Lemma 2.9]{MonterdeELA}.

\begin{lemma}
\label{alphabeta}
Let $T=T(\omega,\eta)$ be a set of twins in $X$. Then $u,v\in T$ if and only if both conditions hold:
\begin{enumerate}
\item $\be_u-\be_v$ is an eigenvector of $M$, and
\item the eigenvalue corresponding to $\be_u-\be_v$ is given by
\begin{equation}
\label{adjalpha}
\theta=
\begin{cases}
 \omega-\eta, &\text{if $M=A$}\\
 \text{deg}(u)-\omega+\eta, &\text{if $M=L$}\\
 \text{deg}(u)+\omega-\eta, &\text{if $M=Q$}.
\end{cases}
\end{equation}
\end{enumerate}
\end{lemma}

The \textit{eigenvalue support} of $u$ with respect to $M$, denoted $\sigma_u(M)$, is the set
\begin{equation*}
\sigma_u(M)=\{\lambda_j:E_j\textbf{e}_u\neq \textbf{0}\}.
\end{equation*}
With respect to $M$, we say that $u$ and $v$ are \textit{cospectral} if $(E_j)_{u,u}=(E_j)_{v,v}$ for each $j$, \textit{parallel} if for each $j$, $E_j\textbf{e}_u=c_j E_j\textbf{e}_v$ for some constant $c_j$, and \textit{strongly cospectral} if $E_j\textbf{e}_u=\pm E_j\textbf{e}_v$ for each $j$. Note that if $u$ and $v$ are strongly cospectral, then $\sigma_u(M)=\sigma_v(M)$ and we can write $\sigma_u(M)=\sigma_{uv}^+(M)\cup \sigma_{uv}^-(M)$, where
\begin{equation*}
\sigma_{uv}^+(M)=\{\lambda_j:E_j\textbf{e}_u=E_j\textbf{e}_v\}\quad \text{and}\quad \sigma_{uv}^-(M)=\{\lambda_j:E_j\textbf{e}_u=-E_j\textbf{e}_v\}.
\end{equation*}
If $u$ and $v$ are twins, then Lemma \ref{alphabeta} implies that $\theta\in\sigma_u(M)$. Next, we state an algebraic characterization and a useful property of twin vertices (see Lemma 2.10 and Corollary 2.11 in \cite{MonterdeELA} resprectively).

\begin{lemma}
\label{aut}
Vertices $u$ and $v$ are twins in $X$ if and only if there exists an involution on $X$ that switches $u$ and $v$ and fixes all other vertices. Moreover, twin vertices are cospectral with respect to $M$.
\end{lemma}

The next result states that $\sigma_{uv}^-(M)$ is a singleton set whenever $u$ and $v$ are twin vertices that are strongly cospectral \cite[Theorem 3.9]{MonterdeELA}.

\begin{lemma}
\label{strcospchar}
Let $T=\{u,v\}$ be a set of twins in $X$, and consider $\theta$ in (\ref{adjalpha}). If $u$ and $v$ are strongly cospectral, then $\sigma_{uv}^-(M)=\{\theta\}$, and $u$ and $v$ cannot be strongly cospectral to any $w\in V(X)\backslash\{u,v\}$.
\end{lemma}

Combining Lemma \ref{strcospchar} with \cite[Theorem 3.4]{MonterdeELA}) yields $|\sigma_{uv}^+(M)|\geq 2$. We also state a result of Kirkland et al.\ about transition matrices of graphs with twin vertices \cite[Corollary 2]{Monterde2022}.

\begin{corollary}
\label{Lan}
Let $T$ be a set of twins in $X$. If $u,v\in T$ with $u\neq v$, then
\begin{equation}
\label{eq21}
|U(t)_{u,u}|^2+\left(|T|-1\right)|U(t)_{u,v}|^2+\sum_{w\notin T}|U(t)_{u,w}|^2=1
\end{equation}
for any $t\in\mathbb{R}$. If we add that $|T|\geq 3$, then $U(t)_{u,u}\neq 0$ for any $t\in\mathbb{R}$.
\end{corollary}

From (\ref{eq21}), if $T$ is a set of twins in $X$ and $u,v\in T$ such that $u\neq v$, then $|U(t)_{u,v}|^2\leq \frac{1}{|T|-1}$ for all $t\in\mathbb{R}$, and this inequality is strict whenever $|T|\geq 3$. This yields a result similar to \cite[Corollary 3]{Monterde2022}.

\begin{corollary}
\label{nopstpgsttw}
Suppose $X$ has $n\geq 3$ vertices and $T$ is a set of twins in $X$. Then no vertex in $T$ can be involved in proper fractional revival with vertex that is not in $T$. Moreover, if $|T|\geq 3$, then any vertex in $T$ cannot be involved in proper fractional revival with any vertex in $X$.
\end{corollary}

The following result is an immediate consequence of Corollary \ref{nopstpgsttw}.

\begin{corollary}
\label{mono}
Let $u$ be a vertex of $X$ with a twin $v$. Then $u$ can only pair with at most one vertex in $X$ for proper fractional revival. Moreover, if $u$ is involved in proper fractional revival, then it must be with $v$.
\end{corollary}

Chan et al.\ have constructed infinite families of graphs where a vertex may be paired with two others for proper FR \cite{Chan2020}. But as Corollary \ref{mono} implies, proper FR involving a vertex with a twin is monogamous.

We say that a subset $S$ of $\sigma_u(M)$ with at least two elements satisfy the \textit{ratio condition} if 
\begin{equation*}
\dfrac{\lambda_p-\lambda_q}{\lambda_j-\lambda_k}\in\mathbb{Q},
\end{equation*}
for all $\lambda_p,\lambda_q,\lambda_j,\lambda_k\in S$ with $\lambda_j\neq \lambda_k$
The following fact is due to Coutinho and Godsil \cite{Coutinho2021} (Corollary 7.3.1 and Theorem 7.6.1) which characterizes periodic vertices.

\begin{theorem}
\label{ratiocon}
Let $X$ be a weighted graph with possible loops. The following statements are equivalent. 
\begin{enumerate}
\item Vertex $u$ is periodic in $X$ with respect to $M$.
\item $\sigma_u(M)=\{\lambda_1,\ldots,\lambda_n\}$ satisfies the ratio condition.
\end{enumerate}
If we add that $\phi(M,t)\in \mathbb{Z}[x]$, then $u$ is periodic if and only if either (i) $\sigma_u(M)\subseteq \mathbb{Z}$, or (ii) there is a square-free integer $\Delta>1$ and an integer $a$ such that each $\lambda_j=\frac{1}{2}\left(a+c_j\sqrt{\Delta}\right)$ for some integer $c_j$, and the difference between any two eigenvalues in $\sigma_u(M)$ is an integer multiple of $\sqrt{\Delta}$.
\end{theorem}

For periodicity, PST and PGST between twins, we refer the reader to \cite{Monterde2022}.


\section{Fractional revival}\label{secFR}

Since proper $(\alpha,\beta)$-FR occurs from $u$ to $v$ if and only if proper $\left(-\frac{\bar{\alpha}\beta}{\bar{\beta}},\beta\right)$-FR occurs from $v$ to $u$ \cite[Proposition 4.1]{Chan2019}, we may say FR between $u$ and $v$ in place of FR from $u$ to $v$. In this section, we characterize proper FR between twins. The following result is a consequence of \cite[Proposition 4.2]{Chan2019}, Lemma \ref{aut}, and \cite[Lemma 4.1]{Godsil2017}.

\begin{lemma}
\label{frpar}
If proper fractional revival occurs between twin vertices $u$ and $v$ in $X$, then they are strongly cospectral.
\end{lemma}

To characterize proper FR between twins, it suffices to consider a set of twins of size two by Lemma \ref{frpar} and Corollary \ref{nopstpgsttw}.

\begin{theorem}
\label{frtwinschar}
Let $T=\{u,v\}$ be a set of twins in $X$ with $\sigma_u(M)=\{\theta,\lambda_1,\ldots,\lambda_r\}$, where $\theta$ is given in (\ref{adjalpha}). If proper $(e^{i\zeta}\cos\gamma,ie^{i\zeta}\sin\gamma)$-fractional revival occurs between $u$ and $v$ at time $\tau$, then
\begin{enumerate}
\item $u$ and $v$ are strongly cospectral with $\sigma_{uv}^+(M)=\{\lambda_1,\ldots,\lambda_r\}$ and $\sigma_{uv}^-(M)=\{\theta\}$, and
\item for any $j$ and $k$, $\tau(\lambda_j-\lambda_k)\equiv 0$ (mod $2\pi$)
and $\tau(\lambda_j-\theta)\equiv 2\gamma$ (mod $2\pi$).
\end{enumerate}
Moreover, the converse holds if and only if $\gamma$ is not an integer multiple of $\pi$.
\end{theorem}

\begin{proof}
Lemmas \ref{strcospchar} and \ref{frpar} prove (1). To prove (2), let $u$ and $v$ be strongly cospectral twins and $\theta,\lambda_1,\ldots,\lambda_r$ be the eigenvalues in $\sigma_u(M)$ with spectral idempotents $E_{\theta},E_1,\ldots,E_r$, respectively, where $\theta$ is given in (\ref{adjalpha}). If $(\alpha,\beta)$-FR occurs between $u$ and $v$, then there is a time $\tau$ and $\alpha,\beta\in\mathbb{C}$ with $|\alpha|^2+|\beta|^2=1$ such that 
\begin{equation}
\label{frtweq}
U(\tau)\be_u=\alpha \be_u+\beta \be_v.
\end{equation}
From Lemma \ref{strcospchar}, $\sigma_{uv}^-(M)=\{\theta\}$, and so using the fact that the spectral idempotents sum to identity, we get
\begin{equation}
\label{fr2}
\begin{split}
\alpha \be_u+\beta \be_v&=\sum_{\lambda\in\sigma_u(M)}(\alpha E_{\theta}\be_u+\beta E_{\theta}\be_v)=(\alpha-\beta) E_{\theta}\be_u+ (\alpha+\beta) \sum_{j=1}^rE_{j}\be_u.
\end{split}
\end{equation}
Combining (\ref{specdecM}), (\ref{frtweq}) and (\ref{fr2}) then gives us
\begin{equation}
\label{fr3}
\alpha-\beta=e^{i\tau\theta}\quad \text{and}\quad \alpha+\beta=e^{i\tau\lambda_j}
\end{equation}
for each $j$. Solving for $\alpha$ and $\beta$, we obtain $\alpha=e^{i\zeta}\cos\gamma$ and $\beta=ie^{i\zeta}\sin\gamma$, where $\gamma,\zeta\in\mathbb{R}$ are given by
\begin{equation}
\label{angle}
\gamma\equiv \tau(\lambda_1-\theta)/2\ (\text{mod}\ \pi)\quad \text{and}\quad \zeta\equiv \tau(\lambda_1+\theta)/2\ (\text{mod}\ \pi).
\end{equation}
As $\alpha-\beta=e^{i(\zeta-\gamma)}$ and $\alpha+\beta=e^{i(\zeta+\gamma)}$, we can rewrite (\ref{fr3}) as
\begin{equation}
\label{fr4}
e^{i(\zeta-\gamma)}=e^{i\tau\theta}\quad \text{and}\quad e^{i(\zeta+\gamma)}=e^{i\tau\lambda_j}.
\end{equation}
Now, (\ref{fr4}) implies that $e^{i\tau(\lambda_j-\theta)}=e^{i2\gamma}$ for each $j$. Thus, (\ref{fr4}) holds if and only if the two equations in (2) hold. The last statement is straightforward.
\end{proof}

The FR in Theorem \ref{frtwinschar} is balanced if and only if $\gamma$ is an odd multiple of $\frac{\pi}{4}$. Next, we determine necessary conditions for proper FR to occur between twins. In what follows, for strongly cospectral twins $u$ and $v$, we let $\sigma_{uv}^+(M)=\{\lambda_1,\ldots,\lambda_r\}$ and $\sigma_{uv}^-(M)=\{\theta\}$, where $r\geq 2$, $\lambda_1>\lambda_2$ and $\theta$ is given in (\ref{adjalpha}).

\begin{corollary}
\label{frtwinschar1}
If $u$ and $v$ are twins in $X$ that admit proper $(e^{i\zeta}\cos\gamma,ie^{i\zeta}\sin\gamma)$-fractional revival at time $\tau$, then $\sigma_{uv}^+(M)$ satisfies the ratio condition. In particular, if $p_j$ and $q_j$ are integers such that $\frac{\lambda_1-\lambda_j}{\lambda_1-\lambda_2}=\frac{p_j}{q_j}$ and $\operatorname{gcd}\{p_j,q_j\}=1$, then
\begin{equation*}
\tau=\frac{2\pi qk}{\lambda_1-\lambda_2}\quad \text{and}\quad \gamma\equiv qk\left(\frac{\lambda_1-\theta}{\lambda_1-\lambda_2}\right)\pi \ (\text{mod}\ \pi)
\end{equation*}
for some integer $k$, where $q=\operatorname{lcm}(q_2,\ldots,q_n)$ and $q\left(\frac{\lambda_1-\theta}{\lambda_1-\lambda_2}\right)$ is not an integer. 
\end{corollary}

\begin{proof}
Let $u$ and $v$ be twins that admit proper FR at time $\tau$. By Theorem \ref{frtwinschar}, $u$ and $v$ are strongly cospectral with $\sigma_{uv}^+(M)=\{\lambda_1,\ldots,\lambda_r\}$ and $\sigma_{uv}^-(M)=\{\theta\}$. The fact that $\sigma_{uv}^+(M)$ satisfies the ratio condition follows directly from  \cite[ Corollary 5.3]{Chan2019}. Now, the equation $\alpha+\beta=e^{it\lambda_j}$ in (\ref{fr3}) holds if and only if $e^{it(\lambda_1-\lambda_j)}=1$ for all $j\in\{2,\ldots,r\}$. Replacing $\sigma_u(M)$ by $\sigma_{uv}^+(M)$ in the proof of Theorem 4.2 in \cite{Monterde2022}, we get that the minimum time $t$ that satisfies $e^{it(\lambda_1-\lambda_j)}=1$ is $t=\frac{2\pi q}{\lambda_1-\lambda_2}$. Thus, the time $\tau$ at which FR occurs between $u$ and $v$ must be an integer multiple of $t$. That is, $\tau=\frac{2\pi qk}{\lambda_1-\lambda_2}$ for some integer $k$. From (\ref{angle}), $\gamma\equiv \frac{1}{2}\tau(\lambda_1-\theta)\ (\text{mod}\ \pi)\equiv qk(\frac{\lambda_1-\theta}{\lambda_1-\lambda_2})\pi \ (\text{mod}\ \pi)$. As proper FR occurs, $\gamma$ is not an integer multiple of $\pi$. Equivalently, $qk(\frac{\lambda_1-\theta}{\lambda_1-\lambda_2})$ is not an integer, which implies that $q(\frac{\lambda_1-\theta}{\lambda_1-\lambda_2})$ is not an integer.
\end{proof}

Observe that the time $\tau$ in Corollary \ref{frtwinschar1} at which proper FR occurs only depends on the eigenvalues in $\sigma_{uv}^+(M)$. For strongly cospectral vertices, this only happens when the vertices are twins. Indeed, if two vertices are strongly cospectral but are not twins, then \cite[Corollary 3.17(1)]{MonterdeELA} yields $|\sigma_{uv}^-(M)|\geq 2$, and \cite[Corollary 5.6]{Chan2019} implies that $\tau$ depends on the eigenvalues in both $\sigma_{uv}^+(M)$ and $\sigma_{uv}^-(M)$.

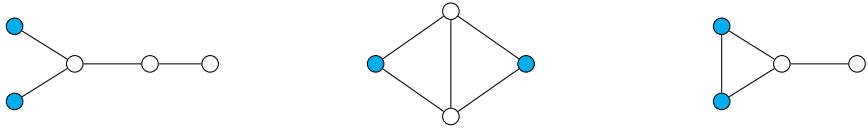
\begin{figure}[h!]
	\begin{center}
		\begin{tikzpicture}
		\tikzset{enclosed/.style={draw, circle, inner sep=0pt, minimum size=.22cm}}
	   
		\node[enclosed, fill=cyan, label={left, yshift=0cm:}] (v_1) at (0,0.5) {};
		\node[enclosed, fill=cyan, label={left, yshift=0cm:}] (v_2) at (0,-0.5) {};
		\node[enclosed] (v_3) at (0.8,0) {};
		\node[enclosed] (v_4) at (1.8,0) {};
		\node[enclosed] (v_5) at (2.6,0) {};
		
		\draw (v_1) --  (v_3);
		\draw (v_2) --  (v_3);
		\draw (v_3) --  (v_4);
		\draw (v_4) --  (v_5);
		
		\node[enclosed, fill=cyan, label={left, yshift=0cm:}] (w_1) at (4.8,0) {};
		\node[enclosed, fill=cyan, label={left, yshift=0cm:}] (w_2) at (6.8,0) {};
		\node[enclosed] (w_3) at (5.8,0.7) {};
		\node[enclosed] (w_4) at (5.8,-0.7) {};
		
		\draw (w_1) --  (w_3);
		\draw (w_1) --  (w_4);
		\draw (w_2) --  (w_3);
		\draw (w_2) --  (w_4);
		\draw (w_3) --  (w_4);
		
		\node[enclosed, fill=cyan, label={left, yshift=0cm:}] (u_1) at (9.4,0.5) {};
		\node[enclosed, fill=cyan, label={left, yshift=0cm:}] (u_2) at (9.4,-0.5) {};
		\node[enclosed] (u_3) at (10.2,0) {};
		\node[enclosed] (u_4) at (11.2,0) {};
		
		\draw (u_1) --  (u_3);
		\draw (u_1) --  (u_2);
		\draw (u_2) --  (u_3);
		\draw (u_3) --  (u_4);
		\end{tikzpicture}
	\end{center}
	\caption{The graphs $X$ (left), $Y$ (center) and $Z$ (right) with twin vertices $u$ and $v$ marked blue}\label{fig}
\end{figure}

We illustrate Corollary \ref{frtwinschar1} with the following example.

\begin{example}
\label{pgst}
Consider the simple unweighted graph $X$ in Fig \ref{fig}. Then $u$ and $v$ are adjacency strongly cospectral with $\sigma_{uv}^+(A)=\left\{\pm\sqrt{2\pm\sqrt{2}}\right\}$ and $\sigma_{uv}^-(A)=\{0\}$. Since $\sigma_{uv}^+(A)$ does not satisfy the ratio condition, proper adjacency FR does not occur between $u$ and $v$ by Corollary \ref{frtwinschar1}. But since all eigenvalues in $\sigma_{uv}^+(A)$ are linearly independent over $\mathbb{Q}$, one checks that adjacency PGST occurs between $u$ and $v$.
\end{example}

The following result not only provides a converse of the first statement in Corollary \ref{frtwinschar1}, but also determines all times $\tau$ and angles $\gamma$ such that $(e^{i\zeta}\cos\gamma,ie^{i\zeta}\sin\gamma)$-FR occurs between twins at time $\tau$.

\begin{corollary}
\label{frtwinschar2}
Let $u$ and $v$ be strongly cospectral twins in $X$ with $\sigma_{uv}^+(M)$ satisfying the ratio condition. Consider $q$ defined in Corollary \ref{frtwinschar1}, and for every positive integer $k$, let
\begin{equation*}
\tau_k=\frac{2\pi qk}{\lambda_1-\lambda_2},\quad \zeta_k\equiv\frac{1}{2}\tau_k(\lambda_1+\theta) \ (\text{mod}\ \pi), \quad  \text{and} \quad \gamma_k\equiv \frac{1}{2}\tau_k(\lambda_1-\theta) \ (\text{mod}\ \pi).
\end{equation*}
Then for every positive integer $k$, $(e^{i\zeta_k}\cos\gamma_k,ie^{i\zeta_k}\sin\gamma_k)$-fractional revival occurs between $u$ and $v$ at time $\tau_k$, which is proper if and only if $qk\left(\frac{\lambda_1-\theta}{\lambda_1-\lambda_2}\right)$ is not an integer. Moreover, $\{\tau_k:k\ \text{is an integer}\}$ and $\{\gamma_k:k\ \text{is an integer}\}$ are the sets of all times $\tau$ and angles $\gamma$ respectively such that $(e^{i\zeta}\cos\gamma,ie^{i\zeta}\sin\gamma)$-fractional revival occurs between $u$ and $v$ at time $\tau$.
\end{corollary}

\begin{proof}
If $\sigma_{uv}^+(M)$ satisfies the ratio condition, then as mentioned in the proof of Corollary \ref{frtwinschar1}, the minimum $t$ such that $e^{it(\lambda_1-\lambda_j)}=1$ for each $j$ is $t=\frac{2\pi q}{\lambda_1-\lambda_2}$. Now, the $\tau_k$, $\zeta_k$ and $\gamma_k$ above satisfies $e^{i(\zeta_k+\gamma_k)}=e^{i\tau_k\lambda_j}$ for all $j\in\{1,\ldots,r\}$ and $e^{i(\zeta_k-\gamma_k)}=e^{i\tau_k\theta}$. Thus, (\ref{fr4}) holds, and so $(e^{i\zeta_k}\cos\gamma_k,ie^{i\zeta_k}\sin\gamma_k)$-FR occurs between $u$ and $v$ at time $\tau_k$. As $\gamma_k\equiv qk(\frac{\lambda_1-\theta}{\lambda_1-\lambda_2})\pi$, $\sin\gamma_k=0$ if and only if $qk(\frac{\lambda_1-\theta}{\lambda_1-\lambda_2})\notin\mathbb{Z}$. As $\gamma_k$ depends on $\tau_k$, and Corollary \ref{frtwinschar1} implies that the time at which FR occurs is an integer multiple of $\tau_1$.
\end{proof}

Thus, the minimum FR time is $\tau_1=\frac{2\pi q}{\lambda_1-\lambda_2}$. If we add that $\phi(M,t)\in\mathbb{Z}[x]$, then $\tau_1=\frac{2\pi}{g\sqrt{\Delta}}$ (see Proposition \ref{q}). Since $\gamma_k=k \gamma_1$, we also note that if $u$ and $v$ are periodic at $\tau_1$, then proper FR does not occur between them, while if $u$ and $v$ admit PST at $\tau_1$, then balanced FR does not occur between them.

Corollaries \ref{frtwinschar1} and \ref{frtwinschar2} altogether yield another characterization of FR between twin vertices.

\begin{theorem}
\label{anotherchar}
Let $u$ and $v$ be twins in $X$. Then proper fractional revival occurs between $u$ and $v$ if and only if these vertices are strongly cospectral, $\sigma_{uv}^+(M)$ satisfies the ratio condition, and $q\left(\frac{\lambda_1-\theta}{\lambda_1-\lambda_2}\right)$ is not an integer for some $\lambda_1,\lambda_2\in \sigma_{uv}^+(M)$, where $q$ is defined in Corollary \ref{frtwinschar1}.
\end{theorem}

We remark that the above characterization holds for any real symmetric $H$ as long as $\theta$ is an eigenvalue for $H$ with eigenvector $\textbf{e}_u-\textbf{e}_v$.

\section{Periodicity, PGST and PST from fractional revival}\label{secFR1}

The following result, which is a combination of Corollaries 5.8 and 5.11 in \cite{Chan2019}, reveals that fractional revival between strongly cospectral vertices implies either periodicity, PST or PGST.

\begin{lemma}
\label{pstpgstfr}
Let $u$ and $v$ be strongly cospectral vertices with $(e^{i\zeta}\cos\gamma,ie^{i\zeta}\sin\gamma)$-fractional revival at $\tau$.
\begin{enumerate}
\item If $\gamma=\frac{a}{b}\pi$ for some integers $a$ and $b$ with $\operatorname{gcd}(a,b)=1$, then $u$ and $v$ are periodic at time $b\tau$. If we add that $b$ is even, then perfect state transfer occurs between $u$ and $v$ at time $b\tau/2$.
\item If $\gamma=c\pi$ for some irrational number $c$, then pretty good state transfer occurs between $u$ and $v$.
\end{enumerate}
\end{lemma}

Recall that balanced $(e^{i\zeta}\cos\gamma,ie^{i\zeta}\sin\gamma)$-FR occurs between $u$ and $v$ if and only if $\gamma=\frac{a\pi}{4}$ for some odd $a$. 
Combining this Lemma \ref{pstpgstfr}(1) yields the following result due to Chan et al.\ \cite[Cor 5.9]{Chan2020}.

\begin{corollary}
\label{pstpgstfrlem}
Let $u$ and $v$ be strongly cospectral vertices. If balanced fractional revival occurs between $u$ and $v$ with minimum time $\tau$, then perfect state transfer occurs between $u$ and $v$ with minimum time $2\tau$.
\end{corollary}

We now apply Lemma \ref{pstpgstfr} to twins. We say that $u$ and $v$ exhibit \textit{proper pretty good state transfer} if these vertices exhibit pretty good state transfer but not periodicity.

\begin{theorem}
\label{pstpgstfr1}
Let $u$ and $v$ be strongly cospectral twins with fractional revival at time $\tau$.
\begin{enumerate}
\item If $\frac{\lambda_1-\theta}{\lambda_1-\lambda_2}=\frac{a}{b}$ for some integers $a$ and $b$ with $\operatorname{gcd}(a,b)=1$, then $u$ and $v$ are periodic at time $b\tau$. If $b$ is even, then perfect state transfer occurs between $u$ and $v$ at time $b\tau/2$.
\item If $\frac{\lambda_1-\theta}{\lambda_1-\lambda_2}\notin\mathbb{Q}$, then proper pretty good state transfer occurs between $u$ and $v$. Moreover, proper $(e^{i\zeta_k}\cos\gamma_k,ie^{i\zeta_k}\sin\gamma_k)$-fractional revival occurs between $u$ and $v$ at time $\tau_k$ for every integer $k$, where $\tau_k$, $\zeta_k$ and $\gamma_k$ are defined in Corollary \ref{frtwinschar2}(1).
\end{enumerate}
\end{theorem}

\begin{proof}
From Corollary \ref{frtwinschar2}(1), we know that $\tau=\tau_k$ and $\gamma=\gamma_k$ for some integer $k$, and so $\gamma\equiv qk\left(\frac{\lambda_1-\theta}{\lambda_1-\lambda_2}\right)\pi$ (mod $\pi$). Thus, if $\frac{\lambda_1-\theta}{\lambda_1-\lambda_2}\in\mathbb{Q}$, then $\gamma$ is a rational multiple of $\pi$, and hence, Lemma \ref{pstpgstfr}(1) applies.
Otherwise, Lemma \ref{pstpgstfr}(2) applies. Moreover, if $\frac{\lambda_1-\theta}{\lambda_1-\lambda_2}\notin\mathbb{Q}$, then $\gamma_{\ell}$ is not a rational multiple of $\pi$ for any integer $\ell$, and so $u$ and $v$ are not periodic at $\tau_{\ell}$ for any integer $\ell$. Applying Corollary \ref{frtwinschar2}(1) completes the proof.
\end{proof}

Suppose proper FR occurs between twins $u$ and $v$. To determine whether periodicity or proper PGST occurs between them, we simply apply Lemma \ref{pstpgstfr} if $\gamma$ is known. However, if $\gamma$ is unknown, but the elements in $\sigma_{uv}^+(M)$ and $\sigma_{uv}^-(M)$ are known, then we may use Theorem \ref{pstpgstfr1}. We also remark that the converses of Theorem \ref{pstpgstfr1} statements (1) and (2) do not hold. That is, twin vertices that admit PGST need not admit proper FR as illustrated by Example \ref{pgst}, and periodic strongly cospectral twin vertices need not exhibit proper FR as illustrated by Example \ref{pgst1}. The latter remark motivates us to characterize periodic strongly cospectral twin vertices that admit proper FR. For an integer $b$, denote the largest power of two that divides $b$ by $\nu_2(b)$. Wealso  write $b\not\divides q$ to denote the fact that $b$ does not divide $q$.

\begin{theorem}
\label{long}
Let $u$ and $v$ be periodic twin vertices that are strongly cospectral. The following hold
\begin{enumerate}
\item Let $a$ and $b$ be integers such that $\frac{\lambda_1-\theta}{\lambda_1-\lambda_2}=\frac{a}{b}$ and $\operatorname{gcd}(a,b)=1$. Proper fractional revival occurs between $u$ and $v$ at time $\tau$ if and only if $b\not\divides q$, where $q$ is defined in Corollary \ref{frtwinschar1}.
\item Suppose $b\not\divides q$ and let $g=\operatorname{gcd}(b,q)$. Then $\{\tau_k:k\ \text{is an integer such that}\ \frac{b}{g}\not\divides k\}$ is the set of all times such that proper fractional revival occurs between $u$ and $v$. The following also hold.
\begin{enumerate}
\item If $\nu_2(q)\geq \nu_2(b)$, then neither perfect state transfer nor balanced fractional revival occurs between $u$ and $v$.
\item If $\nu_2(b)\geq \nu_2(q)+1$, then perfect state transfer occurs between $u$ and $v$ with minimum time $\frac{b\tau_1}{2g}$ and balanced fractional revival occurs between $u$ and $v$ if and only if $\nu_2(b)\geq \nu_2(q)+2$.
\end{enumerate}
\end{enumerate}
\end{theorem}

\begin{proof}
By Theorem \ref{anotherchar}, the above assumption implies that proper FR occurs between $u$ and $v$ if and only if $q\left(\frac{\lambda_1-\theta}{\lambda_1-\lambda_2}\right)\notin\mathbb{Z}$. As $u$ is periodic, Theorem \ref{ratiocon} implies that $\sigma_u(M)$ satisfies the ratio condition, and so $\frac{\lambda_1-\theta}{\lambda_1-\lambda_2}=\frac{a}{b}$ for some integers $a$ and $b$ with $\operatorname{gcd}(a,b)=1$. Thus, $q\left(\frac{\lambda_1-\theta}{\lambda_1-\lambda_2}\right)=\frac{qa}{b}$ is not an integer if and only if $b\not\divides q$. This proves (1). Applying Corollary \ref{frtwinschar2}(1) proves (2), and we note that the $\gamma_k$'s in Corollary \ref{frtwinschar2}(1) satisfy $\gamma_k=k\gamma_1=k\left(\frac{qa}{b}\right)\pi$. Now, suppose $\nu_2(q)\geq \nu_2(b)$. Then any integer multiple of $\gamma_1$ is not an odd multiple of either $\frac{\pi}{4}$ or $\frac{\pi}{2}$, and so neither PST nor balanced FR can occur between $u$ and $v$. This proves (2a). Finally, suppose $\nu_2(b)\geq \nu_2(q)+1$. If $k=\frac{b}{2g}$, then $\gamma_k=\frac{qa}{2g}\pi$ is an odd multiple of $\frac{\pi}{2}$, and so the FR at time $\tau_k$ is PST, and $\tau_k$ is the earliest time that PST occurs because $\gamma_k$ is the smallest positive odd multiple of $\frac{\pi}{2}$. If $\nu_2(b)=\nu_2(q)+1$, $\gamma_{\ell}$ is not an odd multiple of $\frac{\pi}{4}$ for all $\ell$, and so balanced FR does occur between $u$ and $v$. However, if $\nu_2(b)\geq \nu_2(q)+2$, then we could choose $k=\frac{b}{4g}$ so that $\gamma_k=\frac{qa}{4g}\pi$ is an odd multiple of $\frac{\pi}{4}$, so that balanced FR occurs between $u$ and $v$. This proves (2b).
\end{proof}

From Corollary \ref{pstpgstfrlem}, we know that the existence of balanced FR (at time $\tau$) automatically implies the existence of PST (at time $2\tau$). However, the converse of this is not true by virtue of Theorem \ref{long}(2b).

\section{When $\phi(M,t)$ has integer coefficients}\label{secCharPoly}

We now investigate what happens to the results from the previous section if we add that $\phi(M,t)$ has integer coefficients. The following result is an analog of Corollary \ref{frtwinschar1}.

\begin{corollary}
\label{cor1}
Let $\phi(M,t)\in\mathbb{Z}[x]$, and suppose $u$ and $v$ are twins in $X$ that admit fractional revival. Then $\theta$ is an integer and either (i) all elements in $\sigma_{uv}^+(M)$ are integers or (ii) there is a square-free integer $\Delta>1$ and an integer $a$ such that $\lambda_j=\frac{1}{2}(a+c_j\sqrt{\Delta})$ for some integer $c_j$ with the same parity as $a$.
\end{corollary}
\begin{proof}
As $\theta$ is an algebraic integer, it is an integer. Moreover, Corollary \ref{frtwinschar1} implies that $\sigma_{uv}^+(M)$ satisfies the ratio condition, which is equivalent to (i) and (ii) by virtue of \cite[Theorem 7.6.1]{Coutinho2021}.
\end{proof}

We also revisit the parameters $q$ in Corollary \ref{frtwinschar1} and $\tau_k$ in Corollary \ref{frtwinschar2}.

\begin{proposition}
\label{q}
Let $\phi(M,t)\in\mathbb{Z}[x]$. Then $q=\frac{\lambda_1-\lambda_2}{g\sqrt{\Delta}}$ and $\tau_k=\frac{2\pi k}{g\sqrt{\Delta}}$, where $g=\operatorname{gcd}\left(\frac{\lambda_1-\lambda_2}{\sqrt{\Delta}},\ldots, \frac{\lambda_1-\lambda_r}{\sqrt{\Delta}}\right)$, and either $\Delta=1$ if $\sigma_u(M)\subseteq \mathbb{Z}$ or $\Delta>1$ is a square-free integer otherwise.
\end{proposition}

\begin{proof}
As $\phi(M,t)\in\mathbb{Z}[x]$, $\sigma_{uv}^+(M)$ satisfying the ratio condition is equivalent to Corollary \ref{cor1}. Thus, the $q_j$'s in Corollary \ref{frtwinschar1} satisfy $q_j=\frac{\lambda_1-\lambda_2}{h_j\sqrt{\Delta}}$, where $h_j=\operatorname{gcd}(\frac{\lambda_1-\lambda_2}{\sqrt{\Delta}},\frac{\lambda_1-\lambda_j}{\sqrt{\Delta}})$, $\Delta=1$ if $\sigma_u(M)\subseteq \mathbb{Z}$ and $\Delta>1$ is square-free otherwise. Hence, $q=\operatorname{lcm}(\frac{\lambda_1-\lambda_2}{h_2\sqrt{\Delta}},\ldots,\frac{\lambda_1-\lambda_2}{h_r\sqrt{\Delta'}})=\frac{\lambda_1-\lambda_2}{g\sqrt{\Delta}}$, where $g=\operatorname{gcd}(h_2,\ldots,h_r)=\operatorname{gcd}(\frac{\lambda_1-\lambda_2}{\sqrt{\Delta}},\ldots, \frac{\lambda_1-\lambda_r}{\sqrt{\Delta}})$, which in turn yields $\tau_k=\frac{2\pi qk}{\lambda_1-\lambda_2}=\frac{2\pi k}{g\sqrt{\Delta}}$.
\end{proof}

Next, we have the following version of Theorem \ref{anotherchar} whenever $\phi(M,t)\in\mathbb{Z}[x]$.

\begin{theorem}
\label{anothercharA}
Let $\phi(M,t)\in\mathbb{Z}[x]$, and suppose $u$ and $v$ are twins in $X$. Proper fractional revival occurs between $u$ and $v$ if and only if they are strongly cospectral and one of the following conditions hold.
\begin{enumerate}
\item (i) All elements in $\sigma_{uv}^+(M)$ have the form $\frac{1}{2}(2\theta+b_j\sqrt{\Delta})$, where $b_j$ is even and either $\Delta=1$ or $\Delta>1$ is square-free, and (ii) $g\not\divides \frac{\lambda_1-\theta}{\sqrt{\Delta}}$, where $g$ is given in Proposition \ref{q}.
\item All elements in $\sigma_{uv}^+(M)$ have form $\frac{1}{2}(a+b_j\sqrt{\Delta})$, where $a\neq 2\theta$, $b_j$ is even and $\Delta>1$ is square-free.
\end{enumerate}
\end{theorem}

\begin{proof}
Let $u$ and $v$ be strongly cospectral. We show that the two conditions in Theorem \ref{anotherchar} are equivalent to (1) and (2). As $\phi(M,t)\in\mathbb{Z}[x]$, $q=\frac{\lambda_1-\lambda_2}{g\sqrt{\Delta}}$ by Proposition \ref{q}, $\theta$ is an integer by Corollary \ref{cor1}, and $\sigma_{uv}^+(M)$ satisfying the ratio condition is equivalent to Corollary \ref{cor1}(i,ii). If each $\lambda_j\in\sigma_{uv}^+(M)$ have the form $\frac{1}{2}(2\theta+b_j\sqrt{\Delta})$, then $q(\frac{\lambda_1-\theta}{\lambda_1-\lambda_2})=\frac{\lambda_1-\theta}{g\sqrt{\Delta}}\notin\mathbb{Z}$ if and only if $g\not\divides \frac{\lambda_1-\theta}{\sqrt{\Delta}}$. But if each $\lambda_j\in\sigma_{uv}^+(M)$ have the form $\frac{1}{2}(a+b_j\sqrt{\Delta})$, where $a\neq 2\theta$, then $\frac{\lambda_1-\theta}{\lambda_1-\lambda_2}\notin\mathbb{Q}$, and so $q(\frac{\lambda_1-\theta}{\lambda_1-\lambda_2})\notin\mathbb{Z}$. Theorem \ref{anotherchar} completes the proof.
\end{proof}

\begin{example}
\label{pgst2}
Consider the simple unweighted graph $Y$ in Fig \ref{fig}. Then $u$ and $v$ are signless Laplacian strongly cospectral with $\sigma_{uv}^+(Q)=\left\{3\pm\sqrt{5} \right\}$ and $\sigma_{uv}^-(Q)=\{2\}$. By Theorem \ref{anothercharA}(2), proper Laplacian FR occurs between $u$ and $v$ at time $\tau_1=\frac{\pi}{5}$. The same holds for the pair of vertices in $Y$ marked white.
\end{example}

For the Laplacian case, we have the following result, which is consistent with the characterization of proper Laplacian FR obtained by Chan et al.\ \cite[Theorem 26]{Chan2020}.

\begin{theorem}
\label{anothercharL}
Let $\phi(L,t)\in\mathbb{Z}[x]$ and $X$ be a simple positively weighted graph with twins $u$ and $v$. Proper Laplacian fractional revival occurs between $u$ and $v$ if and only if these vertices are strongly cospectral, $\sigma_u(L)\subseteq \mathbb{Z}$, and $g\not\divides \theta$, where $g=\operatorname{gcd}\left(\lambda_1,\lambda_3,\ldots, \lambda_r\right)$ and $\lambda_2=0$.
\end{theorem}

\begin{proof}
The assumption implies that $L$ is positive semi-definite with $0$ as a simple eigenvalue with an all-ones eigenvector. Thus, $0\in \sigma_{uv}^+(M)$, and so Corollary \ref{cor1} implies that the eigenvalues in $\sigma_{uv}^+(M)$ are either all integers, or all integer multiples of $\sqrt{\Delta}$ for some square-free integer $\Delta>1$. As $\sigma_{uv}^+(M)$ is closed under algebraic conjugation, $b\sqrt{\Delta}\in \sigma_{uv}^+(M)$ if and only if $-b\sqrt{\Delta}\in \sigma_{uv}^+(M)$, which cannot happen because $L$ is positive semi-definite. Thus, $\sigma_{uv}^+(M)$ satisfying the ratio condition is equivalent to $\sigma_u(L)\subseteq \mathbb{Z}$. Now, set $\lambda_2=0$. Then $q=\frac{\lambda_1}{g\sqrt{\Delta}}$ by Proposition \ref{q}, and so $q(\frac{\lambda_1-\theta}{\lambda_1-\lambda_2})=\frac{\lambda_1-\theta}{g}\not\in\mathbb{Z}$ if and only if $g\not\divides \theta$, where $g=\operatorname{gcd}\left(\lambda_1,\lambda_3,\ldots, \lambda_r\right)$ because $\sigma_u(L)\subseteq \mathbb{Z}$ and $\lambda_2=0$. Invoking Theorem \ref{anotherchar} completes the proof.
\end{proof}

If $1\in\sigma_{uv}^+(L)$, then proper FR does not occur between $u$ and $v$ (which holds even if they are not twins). Now, if $|\sigma_{uv}^+(L)|=2$, then proper FR occurs between $u$ and $v$ if and only if $\lambda_1\not\divides\theta$ by Theorem \ref{anothercharL}. On the other hand, if $|\sigma_{uv}^+(L)|\geq 3$, then the case that two non-zero eigenvalues in $\sigma_{uv}^+(L)$ are relatively prime yields proper FR, but the case that $g>1$ and either (i) $\theta$ is prime and $\theta\neq g$ or (ii) $\theta=1$ does not. We illustrate Theorem \ref{anothercharL} with the following example.

\begin{example}
\label{pgst1}
Consider the simple unweighted graph $Z$ in Fig \ref{fig}. Then $u$ and $v$ are Laplacian strongly cospectral with $\sigma_{uv}^+(L)=\left\{4,1,0\right\}$ and $\sigma_{uv}^-(L)=\{3\}$. As $1\in \sigma_{uv}^+(L)$, proper Laplacian FR does not occur between $u$ and $v$. However, they are Laplacian periodic with minimum period $2\pi$.
\end{example}

We now give an analog of Theorems \ref{pstpgstfr1} and \ref{long}(2).

\begin{theorem}
\label{sighhh}
Let $\phi(M,t)\in\mathbb{Z}[x]$ and suppose $u$ and $v$ are twins in $X$ that admit fractional revival.
\begin{enumerate}
\item Vertices $u$ and $v$ are periodic if and only if all elements in $\sigma_{uv}^+(M)$ have the form $\frac{1}{2}(2\theta+b_j\sqrt{\Delta})$, where $b_j$ is even and either $\Delta=1$ or $\Delta>1$ is square-free. The following also hold.
\begin{enumerate}
\item If $\nu_2\left(\frac{\lambda_1-\theta}{\sqrt{\Delta}}\right)\geq \nu_2(g)$, then neither PST nor balanced FR occurs between $u$ and $v$.
\item If $\nu_2(g)\geq \nu_2\left(\frac{\lambda_1-\theta}{\sqrt{\Delta}}\right)+1$, then PST occurs between $u$ and $v$ with minimum time $\frac{b\tau_1}{2g}$. Moreover, balanced FR occurs between $u$ and $v$ if and only if $\nu_2(g)\geq \nu_2\left(\frac{\lambda_1-\theta}{\sqrt{\Delta}}\right)+2$.
\end{enumerate}
\item Proper PGST occurs between $u$ and $v$ if and only all elements in $\sigma_{uv}^+(M)$ have the form $\frac{1}{2}(a+b_j\sqrt{\Delta})$, where $a\neq 2\theta$ , $b_j$ is even and $\Delta>1$ is square-free.
\end{enumerate}
\end{theorem}

\begin{proof}
Invoking Corollary \ref{cor1}, either (1$i$) or (2) of Theorem \ref{anothercharA} holds. Since $\theta$ is an integer by Corollary \ref{cor1}, Theorem \ref{ratiocon} yields (1). As $\phi(M,t)\in\mathbb{Z}[x]$, Proposition \ref{q} implies that $q=\frac{\lambda_1-\lambda_2}{g\sqrt{\Delta}}$, and so $q(\frac{\lambda_1-\theta}{\lambda_1-\lambda_2})=\frac{\lambda_1-\theta}{g\sqrt{\Delta}}$. Applying the same argument in Theorem \ref{long}(2) proves (1b). To prove (2), let PGST occur between $u$ and $v$. Since (1$i$) yields periodicity, Theorem \ref{anothercharA}(2) holds. The converse follows from Theorem \ref{pstpgstfr1}(2).
\end{proof}

One can verify that Theorem \ref{sighhh}(1b) is in fact equivalent to the characterization of PST between twins whenever $\phi(M,t)\in\mathbb{Z}[x]$ provided by Kirkland et al. \cite[Theorem 10]{Monterde2022}. Hence, Theorem \ref{sighhh} not only generalizes of known results about periodicity, PST and PGST between twins whenever $\phi(M,t)\in\mathbb{Z}[x]$, but also provides a characterization of balanced FR between twins.

\begin{corollary}
Let $\phi(M,t)\in\mathbb{Z}[x]$ and suppose $u$ and $v$ are twins in $X$ that admit fractional revival. If at least one element in $\sigma_{uv}^+(M)$ is an integer, then $u$ and $v$ are periodic.
\end{corollary}
\begin{proof}
This is a direct consquence of Corollary \ref{cor1} and Theorem \ref{ratiocon}.
\end{proof}

If $X$ is simple and positively weighted, then taking $M=L$ in the above corollary and noting that $0\in \sigma_{uv}^+(L)$ yields the result of Chan et al., which states that periodicity is necessary for proper Laplacian FR to occur \cite[Corollary 31]{Chan2020}. However, for $M\in\{A,Q\}$, the existence of proper FR between (twin) vertices need not imply that they are periodic (see Example \ref{kme} for $M=A$ and Example \ref{pgst2} for $M=Q$).

\section{Double Cones}\label{secDC}

In this section, we characterize the existence of FR between the apexes of double cones. The join of two weighted graphs $X$ and $Y$ is the graph $X\vee Y$ obtained by joining every vertex of $X$ with every vertex of $Y$ with an edge of weight one. If $X\in\{K_2,O_2\}$, then a graph isomorphic to $X\vee Y$ is called a \textit{double cone} on $Y$ and the two vertices of $X$ are called the \textit{apexes} of $X\vee Y$. In particular, $K_2\vee Y$ is called a \textit{connected} double cone, while $O_2\vee Y$ is called a \textit{disconnected} double cone. We start with the following lemma, which is a straightforward consequence of Corollary \ref{frtwinschar2} and Proposition \ref{q}.

\begin{lemma}
\label{dc}
Let $u$ and $v$ be strongly cospectral twins in $X$ with $|\sigma_{uv}^+(M)|=2$. Then $\sigma_{uv}^+(M)$ satisfies the ratio condition, $q=1$ and $\tau_k=\frac{2 k\pi}{\lambda_1-\lambda_2}$. If we add that $\phi(M,t)\in\mathbb{Z}[x]$, then $g=\frac{\lambda_1-\lambda_2}{\sqrt{\Delta}}$ and $\tau_k=\frac{2k\pi}{\sqrt{\Delta}}$.
\end{lemma}

\subsection*{6.1\ \ Disconnected double cones}

We first deal with the Laplacian matrix.

\begin{theorem}
\label{dclap}
Let $Y$ be a simple positively weighted graph on $n$ vertices. For every integer $k>0$ such that $(n+2)\not\divides 2k$, proper Laplacian $(e^{i\zeta_k}\cos\gamma_k,ie^{i\zeta_k}\sin\gamma_k)$-FR occurs between the apexes $u$ and $v$ of $O_2\vee Y$ at time $\tau_k$, where $\tau_k=\gamma_k=\frac{2k\pi}{n+2}$ and $\zeta_k=\frac{2k(n+1)\pi}{n+2}$. The following also hold.
\begin{enumerate}
\item If either $n$ is odd or $n\equiv 0$ (mod 4), then neither PST and nor balanced FR occurs between $u$ and $v$.
\item If $n\equiv 2$ (mod 4), then PST occurs between $u$ and $v$ with minimum time $\frac{\pi}{2}$ and balancedFR occurs between $u$ and $v$ if and only if $\nu_2(n+2)\geq 3$.
\end{enumerate}
\end{theorem}

\begin{proof}
Let $u$ and $v$ be the apexes of $O_2\vee Y$. By \cite[Lemma 5(2)]{Monterde2022}, $\sigma_u(L)=\{0,n,n+2\}$, which applies even if $Y$ is simple and positively weighted. Now, \cite[Corollary 6.9]{MonterdeELA} implies that $u$ and $v$ are strongly cospectral with $\sigma_{uv}^+(L)=\{0,n+2\}$. Invoking Corollary \ref{frtwinschar2} and Lemma \ref{dc} yields $q=1$ and $(e^{i\gamma_k}\cos\gamma_k,ie^{i\gamma_k}\sin\gamma_k)$-FR between $u$ and $v$ at time $\tau_k$ for all integers $k>0$. Since $\frac{\lambda_1-\theta}{\lambda_1-\lambda_2}=\frac{2}{n+2}$ and $q=1$, applying Theorem \ref{long} completes the proof.
\end{proof}

Chan et al.\ were the first to show that any disconnected double cone on $n$ vertices admit proper Laplacian FR between apexes at $\tau_1$. Using the machinery we developed for twins, we add to this result by determining all parameters $\tau$, $\zeta$ and $\gamma$ such that proper Laplacian $(e^{i\zeta}\cos\gamma,ie^{i\zeta}\sin\gamma)$-FR occurs between $u$ and $v$ at time $\tau$, and identifying all values of $n$ which yields balanced FR.

By Theorem \ref{dclap}(2), Laplacian PST occurs between $u$ and $v$ in $O_2\vee Y$ at time $\tau_1$ if and only if $n=2$, a result that was first established in \cite[Theorem 11]{Chan2020}. However, we clarify that for two vertices that exhibit proper Laplacian FR with minimum time $\tau_1$, the nonexistence of Laplacian PST at time $\tau_1$ between then does not imply that Laplacian PST never happens between them. For example, if $n=4p-2$ for some odd $p\geq 3$, then $\nu_2(n+2)=2$. By Theorem \ref{dclap}, proper Laplacian FR occurs between $u$ and $v$ at time $\tau_1=\frac{\pi}{2p}$, while Theorem \ref{dclap}(2) implies that Laplacian PST occurs between them at $p\tau_1=\frac{\pi}{2}$. 

We also mention some interesting results of Chan et al. in \cite{Chan2020}. Let $X$ be a simple unweighted graph on $m$ vertices. They proved that $X$ exhibits proper Laplacian FR at time $\frac{2\pi}{m}$ if and only if $X$ is a disconnected double cone, and they also showed that if $m$ is prime and $X$ admits proper Laplacian FR, then $X$ is a disconnected double cone. We note, however, that if $m$ is composite and $X$ admits proper Laplacian FR, then $X$ need not be a double cone. For example, the cycle $C_6$, which is 2-regular, admits adjacency $(-\frac{1}{2},\frac{\sqrt{3}}{2}i)$-FR between antipodal vertices at time $\frac{2\pi}{3}$ (see \cite[Example 7.3]{Chan2019}).

For $M\in\{A,Q\}$, we assume that $Y$ is a simple unweighted $\ell$-regular graph. Let $u$ and $v$ be the apexes of $O_2\vee Y$. From \cite[Corollary 6.9(1)]{MonterdeELA}, $u$ and $v$ are strongly cospectral, and Lemma \ref{strcospchar} yields $\sigma_{uv}^-(A)=\{0\}$ and $\sigma_{uv}^-(Q)=\{n\}$. Invoking Lemma 4(2b) and Corollary 4(2b) in \cite{Monterde2022} gives us $\sigma_u(A)=\{\lambda^{\pm},0\}$, where $\lambda^{\pm}=\frac{1}{2}(\ell\pm\sqrt{\ell^2+8n})$, and $\sigma_u(Q)=\{\lambda^{\pm},n\}$, where $\lambda^{\pm}=\frac{1}{2}(2\ell+n+2\pm\sqrt{(2\ell+n+2)^2-8\ell n})$, resp.
In both cases, $|\sigma_{uv}^+(M)|=2$. For $M=A$, we have the following.

\begin{theorem}
\label{dcadj}
Let $Y$ be a simple unweighted $\ell$-regular graph on $n$ vertices and $D_1=\ell^2+8n$. Then adjacency $(e^{i\gamma_k}\cos\gamma_k,ie^{i\gamma_k}\sin\gamma_k)$-FR occurs between the apexes of $O_2\vee Y$ at time $\tau_k$ for all integers $k>0$, where $\tau_k=\frac{2k\pi}{\sqrt{D_1}}$ and $\gamma_k=\frac{\left(\ell+\sqrt{D_1}\right)k\pi}{2\sqrt{D_1}}$. The following also hold.
\begin{enumerate}
\item If $\ell=0$, then PST occurs between $u$ and $v$ with minimum time $\tau_1=\frac{\pi}{\sqrt{2n}}$.
\item If $\ell>0$ and $\ell^2+8n$ is not a perfect square, then proper $(e^{i\gamma_k}\cos\gamma_k,e^{i\gamma_k}\sin\gamma_k)$-FR occurs between $u$ and $v$ at each $\tau_k$. Proper PGST also occurs between $u$ and $v$.
\item Let $\ell>0$, and suppose that $n=\frac{1}{2}s(\ell+s)$ for some integer $s$ such that $s(\ell+s)$ is even and $f=\operatorname{gcd}(\ell+s,s)$. Then $\{\tau_k:k\ \text{is an integer such that}\ \frac{\ell+2s}{f}\not\divides k\}$ is the set of all times such that proper fractional revival occurs between $u$ and $v$. The following also hold.
\begin{enumerate}
\item  If $\nu_2(\ell)\leq \nu_2(s)$, then neither PST nor balanced FR occurs between $u$ and $v$.
\item If $\nu_2(\ell)\geq \nu_2(s)+1$, then PST occurs between $u$ and $v$ with minimum time $\frac{\pi}{2}$ and balanced FR occurs between $u$ and $v$ if and only if $\nu_2(\ell)=\nu_2(s)+1$.
\end{enumerate}
\end{enumerate}
\end{theorem}

\begin{proof}
Corollary \ref{frtwinschar2} and Lemma \ref{dc} imply that $q=1$ and $(e^{i\gamma_k}\cos\gamma_k,ie^{i\gamma_k}\sin\gamma_k)$-FR occurs between $u$ and $v$ at time $\tau_k$ for all integers $k>0$. If $\ell=0$, then $\gamma_k=\frac{k\pi}{2}$, and so PST occurs between $u$ and $v$ at time $\tau_k$ for every odd $k$, but balanced FR does not. Thus, (1) holds. Now, let $\ell>0$. If $D_1$ is not a perfect square, then $\frac{\lambda^+-\theta}{\lambda^+-\lambda^+}\notin \mathbb{Q}$. Invoking Theorem \ref{pstpgstfr1}(2) proves (2). Finally, note that $D_1$ is a perfect square if and only if $8n=4s(\ell+s)$ for some integer $s$ such that $s(\ell+s)$ is even. Thus, if $\ell^2+8n$ is a perfect square, then $\frac{\lambda^+-\theta}{\lambda^+-\lambda^+}=\frac{(\ell+s)/f}{(\ell+2s)/f}$, where $f=\operatorname{gcd}(\ell+s,\ell+2s)=\operatorname{gcd}(\ell+s,s)$.
\begin{itemize}
\item Let $\nu_2(\ell)\leq \nu_2(s)$. If $\nu_2(\ell)=\nu_2(s)$,
then $\nu_2(\ell+2s)=\nu_2(\ell)=\nu_2(s)<\nu_2(\ell+s)$, and so $\nu_2(f)=\nu_2(s)=\nu_2(\ell+2s)$. But if $\nu_2(\ell)<\nu_2(s)$, then $\nu_2(\ell+2s)=\nu_2(\ell+s)=\nu_2(\ell)<\nu_2(s)$, and so $\nu_2(f)=\nu_2(\ell+s)=\nu_2(\ell+2s)$. In both cases, we get that $(\ell+2s)/f$ is odd.
\item Let $\nu_2(\ell)\geq \nu_2(s)+1$. Then $\nu_2(\ell+s)=\nu_2(s)\geq 1$ because $s(\ell+s)$ is even, and so $\nu_2(f)=\nu_2(s)$. If $\nu_2(\ell)=\nu_2(s)+1$, then $\nu_2(\ell+2s)\geq \nu_2(s)+2$, and hence, $\nu_2((\ell+2s)/f)\geq 2$. Meanwhile, if $\nu_2(\ell)\geq \nu_2(s)+2 $, then $\nu_2(\ell+2s)=\nu_2(s)+1$, and so $\nu_2((\ell+2s)/f)=1$.
\end{itemize}
Since $q=1$, applying Theorem \ref{long} completes the proof of (3).
\end{proof}

For the signless Laplacian case, we have the following result. Note that we omit the case $\ell=0$ as $O_2\vee Y$ is a bipartite graph and so Theorem \ref{dclap} applies.

\begin{theorem}
\label{dcslap}
Let $Y$ be a simple unweighted $\ell$-regular graph on $n$ vertices, where $\ell>0$ and $D_2=(2\ell+n+2)^2-8\ell n$. Then signless Laplacian $(e^{i\zeta_k}\cos\gamma_k,ie^{i\zeta_k}\sin\gamma_k)$-FR occurs between the apexes of $O_2\vee Y$ at $\tau_k$ for all integers $k>0$, where $\tau_k=\frac{2k\pi}{\sqrt{D_2}}$, $\gamma_k=\frac{\left(2\ell-n+2+\sqrt{D_2}\right)k\pi}{2\sqrt{D_2}}$ and $\zeta_k=\frac{\left(2\ell+3n+2+\sqrt{D_2}\right)k\pi}{2\sqrt{D_2}}$. The following also hold.
\begin{enumerate}
\item If $n=2\ell+2$, then PST between $u$ and $v$ at $\tau_k$ for every odd $k$.
\item If $n\neq 2\ell+2$ and $D_2$ is not a perfect square, then proper $(e^{i\gamma_k}\cos\gamma_k,e^{i\gamma_k}\sin\gamma_k)$-FR occurs between $u$ and $v$ at each $\tau_k$. Proper PGST also occurs between $u$ and $v$.
\item Let $n\neq 2\ell+2$, $n=\frac{s(2\ell-s+2)}{2\ell-s}$ for some integer $s$ such that $s(2\ell+n-s+2)$ is even (i.e., $D_2$ is a perfect square) and $f=\operatorname{gcd}(2\ell-s+2,n-s)$. Then $\{\tau_k:k\ \text{is an integer such that}\ \frac{2\ell+n-2s+2}{f}\not\divides k\}$ is the set of all times such that proper FR occurs between $u$ and $v$.
\begin{enumerate}
\item If $\nu_2(2\ell-s+2)\neq \nu_2(n-s)$, then neither PST nor balanced FR occurs between $u$ and $v$.
\item If $\nu_2(2\ell-s+2)=\nu_2(n-s)$, then PSTr occurs between $u$ and $v$ with minimum time $\frac{\pi}{2}$ and balanced FR occurs between $u$ and $v$ if and only if $\nu_2(a+b)\geq 2$, where $a$ and $b$ are odd integers such that $2\ell-s+2=2^{\nu_2(2\ell-s+2)}a$ and $n-s=2^{\nu_2(n-s)}b$.
\end{enumerate}
\end{enumerate}
\end{theorem}

\begin{proof}
Corollary \ref{frtwinschar2} and Lemma \ref{dc} yield $(e^{i\gamma_k}\cos\gamma_k,ie^{i\gamma_k}\sin\gamma_k)$-FR occurs between $u$ and $v$ at $\tau_k$ for every $k>0$. Moreover, (1) and (2) follow immediately from Theorem \ref{pstpgstfr1}. To prove (3), let $D_2=(2\ell+n+2)^2-8\ell n$ be a perfect square, i.e., $8\ell n=4s(2\ell+n-s+2)$ for some integer $s$ such that $s(2\ell+n-s+2)$ is even. This gives us $n=\frac{s(2\ell-s+2)}{2\ell-s}$, and one checks that $\frac{\lambda^+-\theta}{\lambda^+-\lambda^+}=\frac{2\ell-n+2+\sqrt{D_2}}{2\sqrt{D_2}}=\frac{(2\ell-s+2)/f}{(2\ell+n-2s+2)/f}$, where $f=\operatorname{gcd}(2\ell-s+2,n-s)$. Using the same argument in the proof of Theorem \ref{dcadj} proves (3).
\end{proof}

\subsection*{6.2\ \ Connected double cones}

For connected double cones, we have the following negative result with respect to the Laplacian matrix.

\begin{theorem}
\label{cclap}
Let $Y$ be a simple positively weighted graph on $n$ vertices. If the edge joining the apexes of $K_2\vee Y$ has weight one, then proper FR does not occur between the apexes of $K_2\vee Y$.
\end{theorem}

\begin{proof}
From \cite[Corollary 6.8(2)]{MonterdeELA}, the apexes of $K_2\vee Y$ are not strongly cospectral. Invoking Theorem \ref{frtwinschar} then yields the desired result.
\end{proof}

For the adjacency and signless Laplacian case, we again assume that $Y$ is a simple unweighted $\ell$-regular graph, where $\ell<n-1$. Let $u$ and $v$ be the apexes of $K_2\vee Y$, which are strongly cospectral by \cite[Corollary 6.9(2)]{MonterdeELA}. Lemma \ref{strcospchar} gives us $\sigma_{uv}^-(A)=\{-1\}$ and $\sigma_{uv}^-(Q)=\{n\}$. Invoking Lemma 4(2a) and Corollary 4(2a) in \cite{Monterde2022} then yields $\sigma_u(A)=\{\lambda^{\pm},-1\}$, where $\lambda^{\pm}=\frac{1}{2}(\ell+1\pm\sqrt{D_1})$ and $D_1=(\ell-1)^2+8n$, and
$\sigma_u(Q)=\{\lambda^{\pm},n\}$, where $\lambda^{\pm}=\frac{1}{2}(2\ell+n+4\pm\sqrt{D_2})$ and $D_2=(2\ell-n)^2+8n$, respectively. If $M=A$, then $\frac{\lambda^+-\theta}{\lambda^+-\lambda^+}=\frac{\ell+3\pm\sqrt{D_1}}{2\sqrt{D_1}}$, while if $M=Q$, then $\frac{\lambda^+-\theta}{\lambda^+-\lambda^+}=\frac{2\ell-n+4+\sqrt{D_2}}{2\sqrt{D_2}}$.
Using the same arguments in the proofs of Theorems \ref{dcadj} and \ref{dcslap}, we get the following characterizations of connected double cones on regular graphs admitting proper FR with respect to $M=A$ and $M=Q$.

\begin{theorem}
\label{ccadj}
Let $Y$ be a simple unweighted $\ell$-regular graph on $n$ vertices with $\ell<n-1$ and $D_1=(\ell-1)^2+8n$. Then adjacency $(e^{i\gamma_k}\cos\gamma_k,ie^{i\gamma_k}\sin\gamma_k)$-FR occurs between the apexes of $K_2\vee Y$ at $\tau_k$ for all integers $k>0$, where $\tau_k=\frac{2k\pi}{\sqrt{D_1}}$, $\gamma_k=\frac{\left(\ell+3+\sqrt{D_1}\right)k\pi}{2\sqrt{D_1}}$, and
$\zeta_k=\frac{\left(\ell-1+\sqrt{D_1}\right)k\pi}{2\sqrt{D_1}}$. The following also hold.
\begin{enumerate}
\item If $D_1$ is not a perfect square, then proper $(e^{i\gamma_k}\cos\gamma_k,e^{i\gamma_k}\sin\gamma_k)$-FR occurs between $u$ and $v$ at each $\tau_k$. Proper PGST also occurs between $u$ and $v$.
\item Let $n=\frac{1}{2}s(\ell-1+s)$ for some integer $s$ such that $s(\ell-1+s)$ is even (i.e., $D_1$ is a perfect square) and $f=\operatorname{gcd}(\ell+s+1,s-2)$. Then $\{\tau_k:k\ \text{is an integer such that}\ \frac{\ell+2s-1}{f}\not\divides k\}$ is the set of all times such that proper FR occurs between $u$ and $v$. The following also hold.
\begin{enumerate}
\item If $\nu_2(\ell+s+1)\neq \nu_2(s-2)$, then neither PST nor balanced FR occurs between $u$ and $v$.
\item If $\nu_2(\ell+s+1)=\nu_2(s-2)$, then PST occurs between $u$ and $v$ with minimum time $\frac{\pi}{2}$ and balanced FR occurs between $u$ and $v$ if and only if $\nu_2(a+b)\geq 2$, where $a$ and $b$ are odd integers such that $\ell+s+1=2^{\nu_2(\ell+s+1)}a$ and $s-2=2^{\nu_2(s-2)}b$.
\end{enumerate}
\end{enumerate}
\end{theorem}

\begin{theorem}
\label{ccslap}
Let $Y$ be a simple unweighted $\ell$-regular graph on $n$ vertices with $\ell<n-1$ and $D_2=(2\ell-n)^2+8n$. Then signless Laplacian $(e^{i\gamma_k}\cos\gamma_k,ie^{i\gamma_k}\sin\gamma_k)$-FR occurs between the apexes of $K_2\vee Y$ at $\tau_k$ for all integers $k>0$, where $\tau_k=\frac{2k\pi}{\sqrt{D_2}}$, $\gamma_k=\frac{\left(2\ell-n+4+\sqrt{D_2}\right)k\pi}{2\sqrt{D_2}}$, and $\zeta_k=\frac{\left(2\ell+3n+4+\sqrt{D_2}\right)k\pi}{2\sqrt{D_2}}$. The following also hold.
\begin{enumerate}
\item If $n=2\ell+4$, then PST occurs between $u$ and $v$ at time $\tau_k$ for every odd $k$.
\item If $n\neq 2\ell+4$ and $D_2$ is not a perfect square, then proper $(e^{i\gamma_k}\cos\gamma_k,e^{i\gamma_k}\sin\gamma_k)$-FR occurs between $u$ and $v$ at each $\tau_k$. Proper PGST also occurs between $u$ and $v$.
\item Let $n\neq 2\ell+4$, $n=\frac{s(2\ell+s)}{s+2}$ for some integer $s$ such that $s(2\ell-n+s)$ is even (i.e., $D_2$ is a perfect square) and $f=\operatorname{gcd}(2\ell-n+s+2,s-2)$. Then $\{\tau_k:k\ \text{is an integer such that}\ \frac{2\ell-n+2s}{f}\not\divides k\}$ is the set of all times such that proper FR occurs between $u$ and $v$. The following also hold.
\begin{enumerate}
\item If $\nu_2(2\ell-n+s+2)\neq \nu_2(s-2)$, then neither PST nor balanced FR occurs between $u$ and $v$.
\item If $\nu_2(2\ell-n+s+2)=\nu_2(s-2)$, then PST occurs between $u$ and $v$ at $\frac{\pi}{2}$ and balanced FR occurs between $u$ and $v$ if and only if $\nu_2(a+b)\geq 2$, where $a$ and $b$ are odd integers such that $s-2=2^{\nu_2(s-2)}b$ and $2\ell-n+s+2=2^{\nu_2(2\ell-n+s+2)}a$.
\end{enumerate}
\end{enumerate}
\end{theorem}

Kirkland et al.\ characterized disconnected and connected double cones that admit adjacency and signless Laplacian PST. In fact, one can verify that the conditions in Theorems \ref{dcadj}(1,3b) and \ref{dcslap}(1,3b) are equivalent to the conditions provided in \cite[Theorem 11]{Monterde2022} for disconnected double cones to admit PST, while the conditions in Theorems \ref{ccadj}(2b) and \ref{ccslap}(1,3b) are equivalent to the conditions provided in \cite[Theorem 12]{Monterde2022} for connected double cones to admit PST. We end with the following examples.

\begin{example}
\label{kme}
Let $m\geq 4$ and consider the complete graph minus an edge $K_m\backslash e=O_2\vee K_{m-2}$. We invoke Theorem \ref{dcadj} with $\ell=m-3>0$ and $n=m-2$. As $\ell^2+8n=m^2+2m-7$ is not a perfect square for all $m\geq 4$, proper adjacency $(e^{i\gamma_k}\cos\gamma_k,e^{i\gamma_k}\sin\gamma_k)$-FR between $u$ and $v$ in $K_m\backslash e$ at $\tau_k$ for every $k$, where $\tau_k=\frac{2k\pi}{\sqrt{m^2+2m-7}}$ and $\gamma_k=\frac{\left(m-3+\sqrt{m^2+2m-7}\right)k\pi}{2\sqrt{m^2+2m-7}}$. Proper adjacency PGST also occurs between $u$ and $v$.
\end{example}

\begin{example}
Let $n=s+\sqrt{2s+1}-1$ for some integer $s\equiv 0$ (mod 4) such that $2s+1$ is a perfect square and $\sqrt{2s+1}-1\equiv 2$ (mod 4). Assume $X$ is a simple unweighted $\ell$-regular graph on $n$ vertices, where $\ell=\frac{n}{2}+1$. Since $\nu_2(n)=1$ and $\nu_2(s)\geq 2$, we get $\nu_2(n-s)=1$. Invoking Theorem \ref{dcslap}, we get $f=\operatorname{gcd}(n-s,n-s+4)=\operatorname{gcd}(n-s,4)=2$, and for all integers $k$ such that $n-s+2\not\divides k$, proper signless Laplacian $(e^{i\zeta_k}\cos\gamma_k,ie^{i\zeta_k}\sin\gamma_k)$-FR occurs between $u$ and $v$ at time $\tau_k$ for every $k$, where $\tau_k=\frac{k\pi}{n-s+2}$ and $\gamma_k=\frac{(n-s+4)k\pi}{2(n-s+2)}$. Furthermore, since $\nu_2(n-s+4)=\nu_2(n-s)=2$, signless Laplacian PST occurs between $u$ and $v$ with minimum time $\frac{\pi}{2}$. In particular, if we write $n-s+4=2a$ and $n-s=2b$ for some odd $a$ and $b$, then balanced signless Laplacian FR occurs between $u$ and $v$ if and only if $\nu_2(a+b)\geq 2$.
\end{example}

\section*{Acknowledgements}
H.M.\ is supported by the University of Manitoba Faculty of Science and Faculty of Graduate Studies. H.M. would like to thank Steve Kirkland and Sarah Plosker for the guidance and helpful comments.

\bibliographystyle{alpha}
\bibliography{mybibfile}
\end{document}